\date{November 28, 2011.}
\documentclass[a4paper,10pt,reqno]{amsart}
\usepackage[utf8x]{inputenc}
\usepackage{amsthm}
\usepackage{amsmath,amssymb}
\usepackage{graphicx}
\usepackage{xypic}
\usepackage{mathptmx}

\def\bC{\mathbb{C}}
\def\bF{\mathbb{F}}
\def\bH{\mathbb{H}}
\def\bP{\mathbb{P}}
\def\bQ{\mathbb{Q}}
\def\bR{\mathbb{R}}
\def\bT{\mathbb{T}}
\def\bZ{\mathbb{Z}}

\def\a{\alpha}
\def\b{\beta}
\def\g{\gamma}

\def\f{\varphi}
\def\o{\omega}
\def\s{\sigma}

\def\cC{\mathcal{C}}

\def\cP{\mathcal{P}}
\def\cQ{\mathcal{Q}}
\def\cX{\mathcal{X}}

\def\im{\textrm{Im}}
\def\Gr{\textrm{Gr}}
\def\op{\oplus}
\def\ot{\otimes}

\def\bk{\mathbf{k}}

\def\fg{\mathfrak{g}}

\def\l{\langle}
\def\r{\rangle}

\newtheorem{teo}{Theorem}

\newtheorem{prop}{Proposition}

\theoremstyle{definition} \newtheorem{rem}{Remark}
\newtheorem{lemma}{Lemma}
\theoremstyle{definition} 

\author[G. Bazzoni]{Giovanni Bazzoni}
\address{ICMAT (Instituto de Ciencias Matem\'{a}ticas)
CSIC-UAM-UC3M-UCM, Consejo Superior de Investigaciones Cient\'{i}ficas,
C/ Nicol\'{a}s Cabrera 13-15, Campus Cantoblanco, UAM, 28049 Madrid, Spain}
\email{gbazzoni@icmat.es}

\thanks{Partially supported by Spanish grant MICINN ref.\ MTM2010-17389.}

\title{Minimal algebras and $2-$step nilpotent Lie algebras in dimension 7}

\begin{document}

\begin{abstract}
We use the methods of \cite{BM} to give a classification of $7-$dimensional minimal algebras, generated in degree 1, over any field $\bk$ of characteristic $\textrm{char}(\bk)\neq 2$, whose characteristic filtration has length 2. Equivalently, we classify $2-$step nilpotent Lie algebras in dimension 7. This classification also recovers the real homotopy type of $7-$dimensional $2-$step nilmanifolds.
\end{abstract}

\maketitle

\section{Introduction and Main Theorem}

In this paper we classify some minimal algebras of dimension 7 generated in degree 1 over a field $\bk$ with
$\textrm{char}(\bk)\neq 2$. More specifically, we focus on minimal algebras whose characteristic filtration has
length 2. This recovers the classification of $2-$step nilpotent Lie algebras over $\bk$ in dimension 7. This classification had already been obtained over the fields $\bC$ and $\bR$ (see for instance \cite{Gon}, \cite{GK} or \cite{Ma}), but the result over arbitrary fields is original. When the field $\bk$ has characteristic zero, we obtain a classification of $2-$step nilmanifolds in dimension 7, up to $\bk-$homotopy type. The approach to this classification problem is different from others. Indeed, the starting point is the classification of minimal algebras as examples of homotopy types of nilmanifolds.\\

The main theorem is stated in terms of $7-$dimensional minimal algebras generated in degree 1 of length 2.

\begin{teo}\label{mainthm}
There are $10+2r+s$ isomorphism classes of minimal algebras of dimension 7 and length 2, generated in degree 1, over a field $\bk$ of characteristic different from two; $r$ is the cardinality of the square class group $\bk^*/(\bk^*)^2$ and $s$ is the number of non-isomorphic quaternion algebras over $\bk$. In particular, when $\bk$ is algebraically closed, $r=s=1$ and there 13 non-isomorphic minimal algebras; when $\bk=\bR$, $r=s=2$ and there are 16.
\end{teo}

This paper is organized as follows. In the first section we recall all the relevant algebraic and topological definitions (minimal algebras, nilpotent Lie algebras, nilmanifolds). In the following sections we proceed with the classification, which is accomplished by a case-by-case study. 

\noindent \textbf{Acknowledgements.} The author would like to thank Vicente Mu\~{n}oz for his constant help and Jos\'{e} Ignacio Burgos for useful conversations.

\section{Preliminaries}

A \textit{commutative differential graded algebra} (CDGA, for short) over a field $\bk$ (of characteristic
$\textrm{char}(\bk)\neq 2$) is a graded $\bk$-algebra $A=\oplus_{k\geq 0} A^k$ such that $xy=(-1)^{|x||y|} yx$, for
homogeneous elements $x,y$, where $|x|$ denotes the degree of $x$, and endowed with
a differential $d:A^k\to A^{k+1}$, $k\geq 0$, satisfying the graded Leibnitz rule
\begin{equation}\label{eq:Leibnitz}
d(xy)=(dx) y + (-1)^{|x|} x (dy)
\end{equation}
for homogeneous elements $x,y$. Given a CDGA $(A,d)$, one can compute its cohomology, and the cohomology algebra
$H^*(A)$ is itself a CDGA with zero differential. A CDGA is said to be \textit{connected} if $H^0(A)\cong\bk$. A
\textit{CDGA morphism} between CDGAs $(A,d)$ and $(B,d)$ is an algebra morphism which preserves the degree and
commutes with the differential.

A \textit{minimal algebra} is a CDGA $(A,d)$ of the following form:
\begin{enumerate}
 \item $A$ is the free commutative graded
 algebra $\wedge V$ over a graded vector space $V=\oplus V^i$, 
 \item there exists a collection of generators $\{ x_\tau,
 \tau\in I\}$, for some well ordered index set $I$, such that
 $\deg(x_\mu)\leq \deg(x_\tau)$ if $\mu < \tau$ and each $d
 x_\tau$ is expressed in terms of preceding $x_\mu$ ($\mu<\tau$).
 This implies that $dx_\tau$ does not have a linear part.
\end{enumerate}

We have the following fundamental result:
every connected CDGA $(A,d)$ has a \emph{minimal model}; this means that there
exists a minimal algebra $(\wedge V,d)$ together with a CDGA morphism
$$
\varphi:(\wedge V,d)\to(A,d)
$$
which induces an isomorphism on cohomology. The minimal model of a CDGA over a field $\bk$ of characteristic zero
is unique up to isomorphism. The corresponding result for fields of arbitrary characteristic is not known: in fact, existence is proved in exactly in the same way as for characteristic zero, but the uniqueness is an open question.
For a study of minimal models over fields of arbitrary characteristic, see for instance \cite{Ha}.
In \cite{BM}, uniqueness is proved for minimal algebras generated in degree 1.

The \textit{dimension} of a minimal algebra is the dimension over $\bk$ of the graded vector space $V$. We say that
a minimal algebra is \textit{generated in degree} $k$ if the vector space $V$ is concentrated in degree $k$. In this
paper we will focus on minimal algebras of dimension 7 generated in degree 1.\\

We turn to nilpotent Lie algebras; there is a precise correspondence between minimal algebras
generated in degree 1 and nilpotent Lie algebras.

Given a Lie algebra $\fg$, we define the lower central series of $\fg$ as follows:
$$
\fg^{(0)}=\fg,\quad \fg^{(1)}=[\fg,\fg], \quad \textrm{and}\quad  \fg^{(k+1)}=[\fg,\fg^{(k)}].
$$

A Lie algebra $\fg$ is called \textit{nilpotent} if there exists a positive integer $n$ such that $\fg^{(n)}=\{0\}$. In
particular, the nilpotency condition implies that $\fg^{(1)}\subset\fg^{(0)}$.

\begin{lemma}\label{lem:lower}
If $\fg$ is a nilpotent Lie algebra then $\fg^{(0)}\supset\fg^{(1)}\supset\ldots\supset\fg^{(n)}=\{0\}$.
\end{lemma}
\begin{proof}
As we noticed above, $\fg^{(0)}\supset\fg^{(1)}$. We suppose inductively that $\fg^{(k-1)}\supset\fg^{(k)}$ and show
that $\fg^{(k)}\supset\fg^{(k+1)}$: in fact,
$$
\fg^{(k+1)}=[\fg,\fg^{(k)}]\subset[\fg,\fg^{(k-1)}]=\fg^{(k)}.
$$
\end{proof}

One can form the quotients
\begin{equation}\label{eq:quotient}
E_k=\fg^{(k)}/\fg^{(k+1)}
\end{equation}
and write $\fg=\oplus_kE_k$, but the splitting is not canonical. Nevertheless the numbers $e_k:=\dim(E_k)$ are
invariants of the lower central series. Notice that $e_k=0$ eventually.

A nilpotent Lie algebra is called \textit{$m-$step nilpotent} if $\fg^{(m)}=\{0\}$ and $\fg^{(m-1)}\neq \{0\}$.
Notice that if $\fg$ is $m-$step nilpotent then the last nonzero term of the central series, $\fg^{(m-1)}$, is
contained in the center of $\fg$. In this paper we classify nilpotent Lie algebras in dimension 7 which are $2-$step nilpotent. For more details, see \cite{GK}.

Let $\fg$ be a nilpotent Lie algebra of dimension $n$. It is possible to choose a basis $\{X_1,\ldots,X_n\}$ for
$\fg$, called Mal'cev basis, such that the Lie brackets can be written as follows:
\begin{equation}\label{eq:bracket}
[X_i,X_j]=\sum_{k>i,j}a^k_{ij}X_k.
\end{equation}

Let $\fg$ be a nilpotent Lie algebra and let $\{ X_1,\ldots,X_n\}$ be a Mal'cev basis. Consider the dual vector
space $\fg^*$ with the dual basis $\{ x_1,\ldots,x_n\}$, i.e., $x_i(X_j)=\delta^i_j$. We can endow $\fg^*$ with a
differential $d$, defined according to the Lie bracket structure of $\fg$. Namely, we define

\begin{equation}\label{eq:differential}
dx_k=-\sum_{k>i,j}a^k_{ij}x_i\wedge x_j\,.
\end{equation}

We will usually omit the exterior product sign. $\wedge\fg^*$ is the exterior algebra of $\fg^*$, which we assume to
be a vector space concentrated in degree 1; we extend the differential $d$ to $\wedge\fg^*$ by imposing the graded
Leibnitz rule (\ref{eq:Leibnitz}). The CDGA $(\wedge\fg^*,d)$ is the Chevalley-Eilenberg complex associated
to $\fg$. When $\fg$ is nilpotent, the formula for the differential (\ref{eq:differential}) shows that $(\wedge\fg^*,d)$ is a minimal algebra, according to the above definition. Therefore,
the Chevalley-Eilenberg complex of a nilpotent Lie algebra is a minimal algebra generated in degree 1.

Let $(\wedge V,d)$ be a minimal algebra generated in degree 1; in particular, the case of our interest is when
$(\wedge V,d)=(\wedge\fg^*,d)$ is the Chevalley-Eilenberg complex associated to a nilpotent Lie algebra $\fg$. We
define the following subsets of $V$:
$$
\begin{array}{ccl}
W_0 & = & \ker(d)\cap V\\
W_k & = & d^{-1}(\wedge^2W_{k-1}), \ \mathrm{for} \ k\geq 1 \, .
\end{array}
$$
\begin{lemma}\label{lem:filtration}
For any $k\geq 0$, $W_k\subset W_{k+1}$.
\end{lemma}
\begin{proof}
First notice that $W_0\subset W_1$ since $W_0=d^{-1}(0)$. By induction, suppose that $W_{k-1}\subset W_k$; then we
have
$$
d(W_k)=d(d^{-1}(\wedge^2W_{k-1}))\subset\wedge^2W_{k-1}\subset\wedge^2W_k\, .
$$
This proves that $W_k\subset W_{k+1}$, as required.
\end{proof}

In particular, $W_0\subset W_1\subset\ldots\subset W_m=V$ is an increasing filtration of $V$, which we call
\textit{characteristic filtration}. The \textit{length} of the filtration is, by definition, the least $k$ such that
$W_{k-1}=V$. In general, we will say that a minimal algebra generated in degree 1, $(\wedge V,d)$, has length $n$ if
its characteristic filtration has length $n$. Define
$$
\begin{array}{ccl}
F_0 & = & W_0\\
F_k & = & W_k/W_{k-1} \ \mathrm{for} \ k\geq 1 \, .
\end{array}
$$
Then one can write $V=\oplus_kF_k$, although not in a canonical way. Nevertheless, the numbers $f_k=\dim(F_k)$ are
invariants of $V$. Notice that $f_k=0$ eventually, and the length of the filtration coincides with the least $k$
such that $f_k=0$. In case $(\wedge V,d)=(\wedge\fg^*,d)$ one has $F_k=E_k^*$, where the $E_k$ are defined in
(\ref{eq:quotient}).

The differential
$$
 d:W_{k+1} \longrightarrow \wedge^2(F_0\oplus\ldots\oplus F_k)
 $$
can be decomposed according to the following diagram:
$$
\xymatrix{
 W_{k+1} \ar[r]^d\ar@{->>}[d] & \wedge^2 W_k\ar[r]^(.15){\simeq}\ar@{->>}[d] & \wedge^2(F_0\oplus\ldots\oplus
 F_k)\simeq\wedge^2(F_0\oplus\ldots\oplus F_{k-1})\oplus ((F_0\oplus\ldots\oplus F_{k-1})\otimes F_k)\ar@{->>}[d]\\
 F_{k+1} \ar[r]^(.35){\bar{d}} & \wedge^2 W_k/\wedge^2 W_{k-1}\ar[r]^{\simeq} & (F_0\oplus\ldots\oplus
 F_{k-1})\otimes F_k}
$$
where the map
 $$
 \bar{d}: F_{k+1} \rightarrow  (F_0\oplus\ldots\oplus F_{k-1})\otimes F_k
 $$
is injective.

\begin{lemma}\label{lem:step2length2}
A nilpotent Lie algebra $\fg$ is $n-$step nilpotent if and only if the characteristic filtration $\{W_k\}$ of
$\fg^*$ has length $n$.
\end{lemma}
\begin{proof}
We argue by induction. Suppose that $\fg$ is $1-$step nilpotent. Then $\fg$ is abelian and formula
(\ref{eq:differential}), which relates brackets in $\fg$ with differential in $\fg^*$, says that the differential
$d$ is identically zero on $\fg^*$. Therefore $W_0=\fg^*$ and the characteristic filtration has length 1. The
converse is also clear.
Now assume that $\fg$ is $n-$step nilpotent. Set $\tilde{\fg}:=\fg/\fg^{(n-1)}$; then $\tilde{\fg}$ is an
$(n-1)-$step nilpotent Lie algebra, thus the characteristic filtration of $\tilde{\fg}^*$ has length $n-1$ by the
inductive hypothesis. One has then
$$
\tilde{\fg}^*=\left(\fg/\fg^{(n-1)}\right)^*=\textrm{Ann}(\fg^{(n-1)})
$$
and $\fg^*=\tilde{\fg}^*\op(\fg^{(n-1)})^*$. As we remarked above, this splitting is not canonical, but shows that
the length of the characteristic filtration of $\tilde{\fg}^*$ is $n$. The other way is similar and straightforward.
\end{proof}

To sum up, in order to classify $2-$step nilpotent Lie algebras in dimension 7 we can classify minimal algebras in
dimension 7, generated in degree 1, such that the corresponding filtration has length 2.\\

If $(\wedge\fg^*,d)$ is a minimal algebra generated in degree 1, of length 2, one can write $\fg^*=F_0\op F_1$,
where $d$ is identically zero on $F_0$ and $d:F_1\hookrightarrow\wedge^2F_0$. Given a vector $v\in F_1$, we say that
$dv\in\wedge^2F_0$ is a \textit{bivector}. When $\fg^*$ is 7 dimensional, we must handle the following pairs of
numbers:
$$
(f_0,f_1)=(6,1), \ (5,2) \ \textrm{and} \ (4,3).
$$

There are no other possibilities; for instance $(3,4)$ can not be because $\dim(\wedge^2F_0)=3 \leq 4=\dim(F_1)$ and
there can be no injective map $F_1\to\wedge^2F_0$.

We will make systematic use of the following result:

\begin{lemma} \label{lem:bilinear}
 Let $W$ be a vector space of dimension $k$ over a field $\bk$ whose characteristic is different from $2$.
 Given any element $\f\in \wedge^2 W$, there is a (not unique) basis $x_1,\ldots, x_k$
 of $W$ such that $\f=x_1\wedge x_2 +\ldots + x_{2r-1}\wedge x_{2r}$, for some $r\geq 0$,
 $2r\leq k$. The $2r$-dimensional space $\l x_1,\ldots, x_{2r}\r \subset W$ is well-defined
(independent of the basis).
 \end{lemma}

 \begin{proof}
 Interpret $\f$ as a skew-symmetric bilinear map $W^* \times W^* \to \bk$. Let $2r$ be its rank, and consider
 a basis $e_1,\ldots, e_k$ of $W^*$ such that $\f(e_{2i-1}, e_{2i})=1$, $1\leq i\leq r$,
 and the other pairings are zero. Then the dual basis $x_1,\ldots, x_k$ does the job.
 \end{proof}

Finally, we relate our algebraic classification to the classification of rational homotopy types of $7-$dimensional $2-$step nilmanifolds. The bridge from algebra to topology is provided by rational homotopy
theory. In the seminal paper \cite{S}, Sullivan showed that it is possible to associate to any
nilpotent CW-complex $X$ a CDGA, defined over the rational numbers $\bQ$, which encodes the rational homotopy type
of $X$.

More precisely, let $X$ be a nilpotent space of the homotopy type of a CW-complex of finite type over $\bQ$ (all
spaces considered in this paper are of this kind). A space is \textit{nilpotent} if $\pi_1(X)$ is a nilpotent group
and it acts in a nilpotent way on $\pi_k(X)$ for $k>1$. The \textit{rationalization} of $X$ (see \cite{GM}) is a
rational space $X_\bQ$ (i.e., a space whose homotopy groups are rational vector spaces) together with a map $X\to
X_{\bQ}$ inducing isomorphisms $\pi_k(X)\otimes\bQ\to\pi_k(X_\bQ)$ for $k\geq 1$ (recall that the rationalization of
a nilpotent group is well-defined - see for instance \cite{GM}). Two spaces $X$ and $Y$ have the same \textit{rational homotopy type}
if their rationalizations $X_{\bQ}$ and $Y_{\bQ}$ have the same homotopy type, i.e. if there exists a map
$X_{\bQ}\to Y_{\bQ}$ inducing isomorphisms in homotopy groups. Sullivan constructed a $1-1$ correspondence between
nilpotent rational spaces and isomorphism classes of minimal algebras over $\bQ$:
\begin{equation*}
X \leftrightarrow (\wedge V_X,d)\,.
\end{equation*}
The minimal algebra $(\wedge V_X,d)$ is the \textit{minimal model} of the space $X$.

We recall the notion of $\bk-$homotopy type for a field $\bk$ of characteristic 0, given in \cite{BM}. The $\bk-$minimal model of a space $X$ is $(\wedge V_X\otimes\bk,d)$. We say that $X$ and $Y$ have the same $\bk-$homotopy
type if and only if the $\bk-$minimal models $(\wedge V_X\otimes \bk,d)$ and $(\wedge V_Y\otimes \bk,d)$ are
isomorphic.

A \textit{nilmanifold} is a quotient $N=G/\Gamma$ of a nilpotent, simply connected Lie group by a discrete co-compact
subgroup $\Gamma$, such that the resulting quotient is compact (\cite{OT}). According to Nomizu theorem (\cite{Nomizu}), the minimal model of $N$ is precisely the Chevalley-Eilenberg complex $(\wedge\fg^*,d)$ of the nilpotent Lie algebra $\fg$ of $G$. Here, $\fg^*=\hom(\fg,
\bQ)$. Mal'cev proved that the existence of a basis $\{X_i\}$ of $\fg$ with \emph{rational} structure constants
$a_{jk}^i$ in (\ref{eq:bracket}) is equivalent to the existence of a
co-compact $\Gamma \subset G$. The minimal model of the nilmanifold $N=G/\Gamma$ is
\begin{equation*}
 (\wedge(x_1,\ldots,x_n),d),
 \end{equation*}
where $V=\langle x_1,\ldots, x_n\rangle=\oplus_{i=1}^n \bQ x_i$ is the vector space generated by $x_1,\ldots, x_n$
over $\bQ$, with $|x_i|=1$ for every $i=1,\ldots,n$ and $d x_i$
is defined according to (\ref{eq:differential}). We say that $N=G/\Gamma$ is an $m-$step nilmanifold if $\fg$ is an
$m-$step nilpotent Lie algebra.

From this we see that the algebraic classification of $7-$dimensional minimal
algebras generated in degree 1 of length 2 over a field $\bk$ of characteristic 0 gives the classification of $2-$step nilmanifolds of dimension 7 up to $\bk-$homotopy type. It is important here to remark that the knowledge of explicit examples of nilmanifolds is useful when one wants to endow nilmanifolds with extra geometrical structures; for instance, in dimension 7, one may think of nilmanifolds with a $G_2$ structure (see \cite{CF}).


\section{Case $(6,1)$}

The space $F_0$ is $6-$dimensional and the differential $d:F_1\to\wedge^2F_0$
gives a bivector $\f_7\in\wedge^2F_0$; its only invariant is the rank, which can be 2, 4
or 6. We choose a generator $x_7$ for $F_1$ and generators $x_1,\ldots,x_6$ for $F_0$. According to the above
lemma \ref{lem:bilinear}, we have 3 cases:
\begin{description}
 \item[rank 2] $dx_7=x_1x_2$;
 \item[rank 4] $dx_7=x_1x_2+x_3x_4$;
 \item[rank 6] $dx_7=x_1x_2+x_3x_4+x_5x_6$;
\end{description}
We remark that this description is valid over any field $\bk$ with $\textrm{char}(\bk)\neq 2$.


\section{Case $(5,2)$}\label{case(5,2)}

The space $F_0$ has dimension 5 and $F_1$ has dimension 2. The differential is an injective map
$d:F_1\hookrightarrow\wedge^2F_0$; the latter is a $10-$dimensional vector space. The image of $d$ gives two
linearly indipendent bivectors $\varphi_6$, $\varphi_7$ spanning a plane in $\wedge^2F_0$ or, equivalently, a
line $\ell$ in $\bP^9=\bP(\wedge^2F_0)$. The rank of the bivectors can be 2 or 4. The indecomposable (i.e., rank 2) bivectors in $\wedge^2F_0$ are parametrized by the Grassmannian $\Gr(2,F_0)$ of $2-$planes in $F_0$. Under the Pl\"{u}cker embedding, this Grassmannian is sent to a $6-$dimensional subvariety $\cX\subset \bP^9$ of degree 5.
The algebraic classification problem leads us to the geometric study of the mutual position of a line $\ell$ and the smooth projective variety $\cX$ in $\bP^9$. The next proposition describes the possible cases, assuming that $\bk$ is algebraically closed. The case in which $\bk$ is not algebraically closed will be treated separately.
\begin{prop}\label{prop:line_grassmannian}
 Let $V$ be a vector space of dimension 5 over an algebraically closed field $\bk$. Let $\cX$ denote the
 Pl\"{u}cker embedding of the Grassmannian $\Gr(2,V)$ in $\bP^9=\bP(\wedge^2V)$ and let $\ell\subset\bP^9$ be a
 projective line. Then one and only one of the following possibilities occurs:
 \begin{enumerate}
  \item the line $\ell$ and $\cX$ are disjoint;
  \item the line $\ell$ is contained in $\cX$;
  \item the line $\ell$ is tangent to $\cX$;
  \item the line $\ell$ is bisecant to $\cX$.
 \end{enumerate}
\end{prop}
\begin{proof}
As we said before, $\cX$ is a $6-$dimensional smooth subvariety of $\bP^9$ of degree 5; by degree and dimension,
a generic $\bP^3$ cuts $\cX$ in 5 points, but a generic $\bP^2$ need not meet it. The same is also
clearly true for a generic line $\ell$. Thus there are lines in $\bP^9$ disjoint from $\cX$.

Let $W\subset V$ be a 4 dimensional vector subspace. This gives embeddings $\bP^5=\bP(\wedge^2
W)\hookrightarrow\bP^9=\bP(\wedge^2V)$ and $\Gr(2,W)\hookrightarrow \cX$. The Grassmannian $\Gr(2,W)$ is a smooth
quadric in $\bP^5$, and has the property that through any point there are two $2-$planes contained in it. In
particular, $\Gr(2,W)$ contains a line $\ell$, and so does $\cX$. On the other hand, if $\ell$ is contained in
this $\bP^5$ then, by dimension and degree reasons, it cuts the quadric $\Gr(2,W)$, and hence $\cX$, in two
points.

Let $p\in\cX$ be a point and consider the projective tangent space to $\bT_p\cX$. If the line $\ell$ is contained
in this $\bP^6$, and $p\in\ell$, but $\ell$ is not contained in $\cX$ (such a line exists because $\cX$ is not
linear), then $\ell$ is tangent to $\cX$.

To conclude, we show that there are no trisecant lines to $\cX$. Indeed, suppose that a line
$\ell\subset\bP^9$ cuts the Grassmannian in three points. We may assume that $\ell$ is the projectivization of a
vector subspace $U\subset\wedge^2V$ of dimension 2, spanned by bivectors $\phi_1$ and $\phi_2$ such that $\bP(\phi_1)$ and $\bP(\phi_2)$
are two of the three points of intersection of $\ell$ with $\cX$; then the rank of the bivectors $\phi_1$ and $\phi_2$ is 2
and they give two $2-$planes $\pi_1$ and $\pi_2$ in $V$. The fact that there is a third intersection point
between $\ell$ and $\cX$ means that there exists exactly one linear combination $a\phi_1+b\phi_2$, with $a,b\in\bk^*$,
which has rank 2, while all the other linear combination have rank 4. But the planes $\pi_1$ and $\pi_2$
either meet in the origin or they intersect in a line. In the first case, all linear combinations $a\phi_1+b\phi_2$,
$a,b\in\bk^*$, have rank 4, in the second one they have all rank 2.
\end{proof}
\vskip 0.3 cm


\subsection{$\ell\cap\cX=\emptyset$}\label{empty(5,2)} The two bivectors have rank 4. If $\langle\f_6,\f_7\rangle$ is a basis of $\im(d)\subset\wedge^2F_0$, then $\f_j$ is a symplectic form on some $4-$plane $H_j\subset F_0$, $j=6,7$ (here we are somehow identifying $F_0$ with its dual, but this is not a problem, since all the vectors are defined modulo scalars). Suppose first that $H_6=H_7$; then we have two rank 4 bivectors on a $4-$dimensional vector space $H:=H_6$; consider the inclusion $H\hookrightarrow F_0$, which gives $\wedge^2 H\hookrightarrow\wedge^2 F_0$ and, projectivizing, $\bP(\wedge^2H)\hookrightarrow\bP(\wedge^2F_0)$. The rank 2 bivectors in $\wedge^2 H$ are parametrized by the Grassmannian $\Gr(2,H)$ which, as we noticed above, is a quadric hypersurface in $\bP(\wedge^2 H)$. The two bivectors $\f_6$ and $\f_7$ give a projective line $\ell$ contained in $\bP(\wedge^2H)$. For dimension reasons, any line in $\bP(\wedge^2 H)$ meets this quadric
hypersurface\footnote{Here we are using the fact that $\bk$ is
algebraically closed}; therefore we can always choose coordinates in $H$ in such a way that at least on bivector
has rank 2. But our hypothesis is that both bivectors have rank 4 and this implies that $H_6\neq H_7$. We set
$V=H_6\cap H_7$; the Grassmann formula says that $\dim(V)=3$. Notice that $(H_6,\f_6)$ and $(H_7,\f_7)$ are
$4-$dimensional symplectic vector spaces.
\begin{lemma}
If $(W,\o)$ is a symplectic vector space and $U\subset W$ is a codimension 1 subspace, then $U$ is coisotropic,
i.e., the symplectic orthogonal $U^{\o}$ of $U$ is contained in $U$.
\end{lemma}
\begin{proof}
The dimension of $U^{\o}$ is 1. If $U^{\o}\nsubseteq U$ we can write $W=U\oplus
U^{\o}$ for dimension reasons. But this is impossible, because $\o$ would descend to a symplectic form on
$U^{\o}$.
\end{proof}

This shows that $V$ is a coisotropic subspace of both $H_6$ and $H_7$. The differential $d$ gives a map $h:F_1\to F_0/V$, defined up to nonzero scalars; we choose vectors $v_6$ and $v_7$ spanning $F_0/V$ and set $x_j=h^{-1}(v_j)$, $j=6,7$. We choose generators $x_1$, $x_2$ and $x_3$ for $V$ and rename $v_6=x_4$, $v_7=x_5$. Thus we get
$$
H_6=\l x_1,x_2,x_3,x_4\r, \quad H_7=\l x_1,x_2,x_3,x_5\r.
$$
We can write $\f_6=x_1x_2+x_3x_4$. This choice implies that the plane $\pi=\l x_1,x_2\r\subset V$ is symplectic for $\f_6$. If it was also symplectic for $\f_7$, we could write
$$
dx_6=\f_6=x_1x_2+x_3x_4 \quad \textrm{and} \quad dx_7=\f_7=x_1x_2+x_3x_5.
$$
But then setting $x_4'=x_4-x_5$, the bivector $\f'=\f_6-\f_7=x_3(x_4-x_5)=x_3x_4'$ would have rank 2, and this is
not possible. The plane $\pi$ must therefore be Lagrangian for $\f_7$ and consequently $\f_7=x_1x_3+x_2x_5$. This
gives finally
\begin{equation*}
\left\{
\begin{array}{ccl}
dx_6 & = & x_1x_2+x_3x_4\\
dx_7 & = & x_1x_3+x_2x_5
\end{array}
\right.
\end{equation*}


\subsection{$\ell\subset\cX$} This means that both $\f_6$ and $\f_7$ have rank 2. They give two
planes $\pi_6$ and $\pi_7$ in $F_0$, which can not coincide: either their intersection is just the origin, or they share a line. But the first case does not show up; indeed, in that case we could take
coordinates $\{x_1,\ldots,x_5\}$ in $F_0$ so that $dx_6=x_1x_2$ and $dx_7=x_3x_4$. Then all bivectors
$a\f_6+b\f_7$, $ab\in\bk^*$, would have rank 4, contradicting the assumption that $\ell\subset\cX$. This implies that
$\pi_6\cap\pi_7$ is a line, which we suppose spanned by a vector $x_1$. We complete this to a basis $\langle
x_1,x_2\rangle$ of $\pi_6$ and $\langle x_1,x_3\rangle$ of $\pi_7$, giving at the end
\begin{equation*}
\left\{
\begin{array}{ccl}
dx_6 & = & x_1x_2\\
dx_7 & = & x_1x_3
\end{array}
\right.
\end{equation*}


\subsection{$\ell\cap\cX=\{p,q\}$}\label{bisecant(5,2)} This case is complementary to case $\ell\subset\cX$
above. In fact, we still have two rank 2 bivectors $\f_6$ and $\f_7$, but every linear combination
$a\f_6+b\f_7$, $ab\in\bk^*$, must now have rank 4. Thus, arguing as we did there, we exclude the case in which
the $2-$planes associated by $\f_6$ and $\f_7$ intersect in a line and conclude that they intersect in the origin. Then
the expression of the differentials is
\begin{equation*}
\left\{
\begin{array}{ccl}
dx_6 & = & x_1x_2\\
dx_7 & = & x_3x_4
\end{array}
\right.
\end{equation*}


\subsection{$\ell\cap\cX=\{p\}$} In this case the line $\ell$ is tangent to $\cX$. The point
$p$ identifies a rank 2 bivector in $\wedge^2F_0$, while all the other bivectors on $\ell$ have rank 4. This
gives a symplectic $2-$plane $(\pi_6,\f_6)$ and a symplectic $4-$plane $(\pi_7,\f_7)$ in $F_0$. $\pi_6$ can
not be contained in $\pi_7$ as a symplectic subspace; in fact, if this was the case, we could choose coordinates
$\{x_1,\ldots,x_4\}$ in $\pi_7$ in such a way that $\pi_6=\l x_1,x_2\r$, $\f_6=x_1x_2$ and $\f_7=x_1x_2+x_3x_4$;
but then the bivector $\f'=\f_7-\f_6$ would belong to $\ell$ and have rank 2, which is impossible since $\ell$
containes only one rank 2 bivector. Then either $\pi_6\subset\pi_7$ as a Lagrangian subspace, or Grassmann's
formula says that $\dim(\pi_6\cap\pi_7)=1$ and the subspaces meet along a line. In the first case we choose
vectors $x_1,x_2,x_3,x_4$ spanning $\pi_7$; then we can write
\begin{equation*}
\left\{
\begin{array}{ccl}
dx_6 & = & x_1x_2\\
dx_7 & = & x_1x_3+x_2x_4
\end{array}
\right.
\end{equation*}

In the second case, call $x_1$ a generator of this line. We can complete this to a basis of $\pi_6$ and to a
basis of $\pi_7$. In particular, we set
$$
\pi_6=\langle x_1,x_2\rangle \quad \textrm{and} \quad \pi_7=\langle x_1,x_3,x_4,x_5\rangle
$$
and we obtain the following expression for the differentials:
\begin{equation*}
\left\{
\begin{array}{ccl}
dx_6 & = & x_1x_2\\
dx_7 & = & x_1x_3+x_4x_5
\end{array}
\right.
\end{equation*}


\subsection{$\bk$ non algebraically closed} Finally we discuss the case in which the
field $\bk$ is non-algebraically closed. Going through the above list, one sees that there are two points where
the field comes into play. More specifically, in case (4) of proposition \ref{prop:line_grassmannian} above, it
could happen that $\ell$ and $\cX$ intersect in two points with coordinates in the
algebraic closure of $\bk$. As this intersection is invariant by the Galois group, there must be a quadratic
extension $\bk'\supset\bk$ where the coordinates of the two points lie; the two points are conjugate by the
Galois automorphism of $\bk'|\bk$. Therefore, there is an element $a\in\bk^*$ such that $\bk'=\bk(\sqrt{a})$,
$a$ is not a square in $\bk$, and the differentials
$$
dx_6=x_1x_2, \qquad dx_7=x_3x_4.
$$
satisfy that the planes $\pi_6=\l x_1,x_2 \r$ and $\pi_7=\l x_3,x_4 \r$ are conjugate under the
Galois map $\sqrt{a}\mapsto -\sqrt{a}$. Write:
 \begin{eqnarray*}
  x_1 &=& y_1 + \sqrt{a} y_2, \\
  x_2 &=& y_3 + \sqrt{a} y_4, \\
  x_3 &=& y_1 - \sqrt{a} y_2, \\
  x_4 &=& y_3 - \sqrt{a} y_4, \\
  x_5 &=& y_5\\
  x_6 &=& y_6 + \sqrt{a} y_7, \\
  x_7 &=& y_6 - \sqrt{a} y_7,
 \end{eqnarray*}
where $y_1,\ldots, y_7$ are defined over $\bk$. Then $dy_6=y_1y_3+ay_2y_4$, $dy_7=y_1y_4+y_2y_3$. This is the
canonical model. Two of these minimal algebras are not isomorphic over $\bk$ for different quadratic field
extensions, since the equivalence would be given by a $\bk$-isomorphism, therefore commuting with the action of
the Galois group. The quadratic field extensions are parametrized by elements $a\in \bk^*/(\bk^*)^2- \{1\}$.
Note that for $a=1$, setting $z_6=y_6+y_7$ and $z_7=y_6-y_7$, we recover case (4) of proposition
\ref{prop:line_grassmannian}, where $dz_6=(y_1+y_2)(y_3+y_4)$ and $dz_7=(y_1-y_2)(y_3-y_4)$ are of
rank $2$. The model in this case is
\begin{equation}\label{non_alg_closed}
\left\{
\begin{array}{ccl}
dx_6 & = & x_1x_3+ax_2x_4\\
dx_7 & = & x_1x_4+x_2x_3
\end{array}
\right., \quad a\in \bk^*/(\bk^*)^2- \{1\}.
\end{equation}

The other point where the field comes into play is in subsection (\ref{bisecant(5,2)}). There,
in order to exclude the possibility $H_6=H_7$, we used the fact that $\bk$ is algebraically closed. Again, if
$\bk$ is not algebraically closed, we can argue as above and deduce that there exists a quadratic extension
$\bk''=\bk(\sqrt{b})$ with $b$ a nonsquare in $\bk$, such that the two intersection points are interchanged by
the action of the Galois automorphism of $\bk''|\bk$. The model in this case coincides with
(\ref{non_alg_closed}).


\section{Case $(4,3)$}

In this case $F_0$ has dimension 4, $F_1$ has dimension 3 and the differential $d:F_1\hookrightarrow\wedge^2F_0$
determines three linearly independent bivectors $\f_5$, $\f_6$ and $\f_7$ in $\wedge^2F_0$, spanning a
$3-$dimensional vector subspace $d(F_1)\subset \wedge^2F_0$; the rank of the bivectors can be 2 or 4. Taking the
projectivization, we obtain a projective plane $\pi=\bP(d(F_1))\subset\bP(\wedge^2F_0)$. The
indecomposable bivectors in $\wedge^2F_0$ are parametrized by the Grassmannian $\Gr(2,F_0)$ of
$2-$planes in $F_0$. Under the Pl\"{u}cker embedding, this Grassmannian is sent to a quadric hypersurface
$\cQ\subset \bP^5$, known as Klein quadric. As it happened in the previous section, the algebraic classification
problem leads us to the geometric study of the mutual position of a plane $\pi$ and the Klein quadric $\cQ$ in
projective space $\bP^5$. In the next lemmas we study this geometry, assuming that $\bk$ is algebraically
closed. The case in which $\bk$ is non-algebraically closed will be treated separately.

In what follows, we fix a $4-$dimensional vector space $V$ over an algebraically closed field $\bk$ and we
denote by $\cQ$ the Pl\"{u}cker embedding of the Grassmannian $\Gr(2,4)$ in projective space
$\bP^5=\bP(\wedge^2V)$.

\begin{lemma}\label{lem:Klein_quadric_1}
Let $p\in\cQ$ be a point; there exist two planes $\pi_1$ and $\pi_2$ such that $\pi_1\cap\pi_2=\{p\}$ and contained
in $\cQ$.
\end{lemma}
\begin{proof}
We take homogeneous coordinates $[X_0:\ldots:X_5]$ in $\bP^5$. The Klein quadric $\cQ$ is given as the zero locus of the homogeneous quadratic equation $X_0X_5-X_1X_4+X_2X_3$. Since $\cQ$ is homogeneous, we can assume that $p$ is the point $[1:0:0:0:0:0]\in\cQ$; the planes $\pi_1$ and $\pi_2$ have equations $X_2=X_4=X_5=0$ and $X_1=X_3=X_5=0$.
\end{proof}

\begin{lemma}\label{lem:Klein_quadric_2}
Let $p\in\cQ$ be a point and let $\bT_p\cQ\cong\bP^4$ be the projective tangent space to $\cQ$ at $p$. Let $\pi\subset\bT_p\cQ$ be a $2-$plane, with $p\in\pi$. Then one of the following possibilities occurs:
 \begin{enumerate}
 \item $\pi\subset\cQ$;
 \item $\pi\cap\cQ$ is a double line;
 \item $\pi\cap\cQ$ is a pair lines.
 \end{enumerate}
\end{lemma}
\begin{proof}
Take homogeneous coordinates $[X_0:\ldots:X_5]$ in $\bP^5$; as above, the Klein quadric is the zero locus of
the quadratic equation $X_0X_5-X_1X_4+X_2X_3$. We can assume again that $p=[1:0:\ldots:0]$. The tangent space
$\bT_p\cQ\cong\bP^4$ has equation $X_5=0$ and intersects $\cQ$ along the quadric $X_1X_4-X_2X_3=0$. Its rank is 4, thus it is
a cone over a smooth quadric $\cC$ in $\bP^3$, with vertex in $p$. The equation of $\cC$ is $X_1X_4-X_2X_3=0$ in
this $\bP^3=\{X_0=X_5=0\}$; then $\cC\cong\bP^1\times\bP^1$ under the Segre embedding, and it contains a line. The plane $\pi$
intersects this $\bP^3$ in a line $\ell$, which can be contained in $\cC$, or tangent to $\cC$ or bisecant to $\cC$.
In the first case, the whole plane $\pi$ is contained in the quadric $\cQ$, since $\cQ$ contains $\ell$, the point
$p$ and all the lines joining $p$ to $\ell$. In the second case $\pi\cap\cQ$ is a double line; indeed, the cone over
$\cC$ intersected with $\pi$ is just one line, counted with multiplicity. In the third case, $\pi$ contains the cone
over two points, which is a pair of lines.
\end{proof}

These two lemmas cover the cases in which the $2-$plane is in special position. The general case (i.e., the case of a \textit{generic} projective plane in $\bP^5$) is that the intersection between the plane and the Klein quadric is a smooth conic. We collect these results in the next proposition:

\begin{prop}\label{prop:plane_grassmannian}
Let $V$ be a vector space of dimension 4 over an algebraically closed field $\bk$. Let $\cQ$ denote the
Pl\"{u}cker embedding of the Grassmannian $\Gr(2,V)$ in $\bP^5=\bP(\wedge^2V)$ and let $\pi\subset\bP^5$ be a
projective plane. Then one and only one of the following possibilities occurs:
 \begin{enumerate}
  \item the plane $\pi$ is contained in $\cQ$;
  \item the plane $\pi$ is tangent to $\cQ$, and $\pi\cap\cQ$ is either a double line or two lines;
  \item the plane $\pi$ cuts $\cQ$ along a smooth conic.
 \end{enumerate}
\end{prop}

According to this proposition, we study the various cases.


\subsection{$\pi\subset\cQ$} Let $V$ be a vector space of dimension 4 over the field $\bk$. Recall that the Pl\"{u}cker embedding maps $\Gr(2,V)$ onto the Klein quadric $\cQ\subset\bP(\wedge^2 V)$. In the previous section we proved that given a point $p\in\cQ$ there exist two skew planes $\bP^2$ contained in $\cQ$ and such that $p$ belongs to both. Now we describe these planes more precisely.

\begin{lemma}\label{lem:Harris}
Let $\ell\subset V$ be a line and denote by $\Sigma_{\ell}\subset\Gr(2,V)$ the locus of $2-$planes in $V$ containing $\ell$; given a hyperplane $W\subset V$, we denote with $\Sigma_W\subset\Gr(2,V)$ the locus of $2-$planes in $V$ contained in $W$. Under the Pl\"{u}cker embedding $\Sigma_{\ell}$ and $\Sigma_W$ are carried to projective $2-$planes $\bP^2\subset \cQ$; conversely, every projective $2-$plane $\bP^2\subset\cQ$ is equal to the image under the Pl\"{u}cker embedding of either $\Sigma_{\ell}$ or $\Sigma_W$.
\end{lemma}
\begin{proof}
Let us start with the first case. We fix a line $\ell\subset V$ and take a hyperplane $U$ such that $\ell\oplus U=V$. A $2-$plane must intersect $U$ along a line is a line $r$ and then $\Sigma_{\ell}$ is in bijection with the space of lines in $U$, which is a projective plane $\bP^2$. The other case is easier: we have an inclusion $\Sigma_W\hookrightarrow\Gr(2,V)$, and $\Sigma_W$ is a projective plane $\bP^2$ (more precisely, $(\bP^2)^*$). The converse is also easy to see.
\end{proof}

If the projective plane $\pi$ is contained in the quadric, the three bivectors $\f_5$, $\f_6$ and $\f_7$ have rank 2 and
any linear combination of them also has rank 2. They give three planes $\pi_5$, $\pi_6$ and $\pi_7$ in $F_0$.
According to lemma \ref{lem:Harris}, we have two possibilities:
\begin{itemize}
\item $\pi$ is associated to $2-$dimensional vector subspaces of $F_0$ containing a given line $r\subset F_0$. In this case, we choose a vector $x_1$ spanning $r$ and complete it to a basis of each plane, obtaining $\pi_5=\l x_1,x_2\r$, $\pi_6=\l x_1,x_3\r$ and $\pi_7=\l x_1,x_4\r$. In term of differentials,
\begin{equation*}
 \left\{
 \begin{array}{ccl}
 dx_5 & = & x_1x_2\\
 dx_6 & = & x_1x_3\\
 dx_7 & = & x_1x_4
 \end{array}
 \right.
\end{equation*}
\item $\pi$ is associated to $2-$dimensional vector subspaces of $F_0$ contained in a given hyperplane $W\subset F_0$. We can take coordinates so that $W=\l x_1,x_2,x_3\r$ and set $\pi_5=\l x_1,x_2\r$, $\pi_6=\l x_1,x_3\r$ and $\pi_7=\l x_2,x_3\r$. This gives the model
\begin{equation*}
 \left\{
 \begin{array}{ccl}
 dx_5 & = & x_1x_2\\
 dx_6 & = & x_1x_3\\
 dx_7 & = & x_2x_3
 \end{array}
 \right.
\end{equation*}
\end{itemize}


\subsection{$\pi\cap\cQ$ is a double line}\label{double(4,3)}

We can suppose that $\f_5$ and $\f_6$ are on $\ell$, but $\f_7$ is not (recall that the three points can not be collinear). Then every linear combination $a\f_5+b\f_6$ has rank 2 and, arguing as above, the corresponding planes $\pi_5$ and $\pi_6$ in $F_0$ intersect along some line $r\subset F_0$. Since $\f_7\notin\ell$, it has rank 4 and it is then a symplectic form in $F_0$. The lines $\ell_5$ and $\ell_6$, joining $\f_7$ with $\f_5$ and $\f_6$ respectively, are tangent to the Klein quadric $\cQ$, thus their points are bivectors of rank 4 except for $\f_5$ and $\f_6$. Arguing as in case $\ell\cap\cX=\{p\}$ of section (\ref{empty(5,2)}), we deduce that the planes $\pi_5$ and $\pi_6$ are Lagrangian for the symplectic form $\f_7$ and we can choose coordinates in $F_0$ to arrange $\f_5=x_1x_2$, $\f_6=x_1x_3$ and $\f_7=x_1x_4+x_2x_3$. This gives the model
\begin{equation*}
 \left\{
 \begin{array}{ccl}
 dx_5 & = & x_1x_2\\
 dx_6 & = & x_1x_3\\
 dx_7 & = & x_1x_4+x_2x_3
 \end{array}
 \right.
\end{equation*}


\begin{figure}[h!]
\begin{center}
    \caption{The two incident lines in the tangent plane}\label{fig:1}
    \includegraphics[width=10cm]{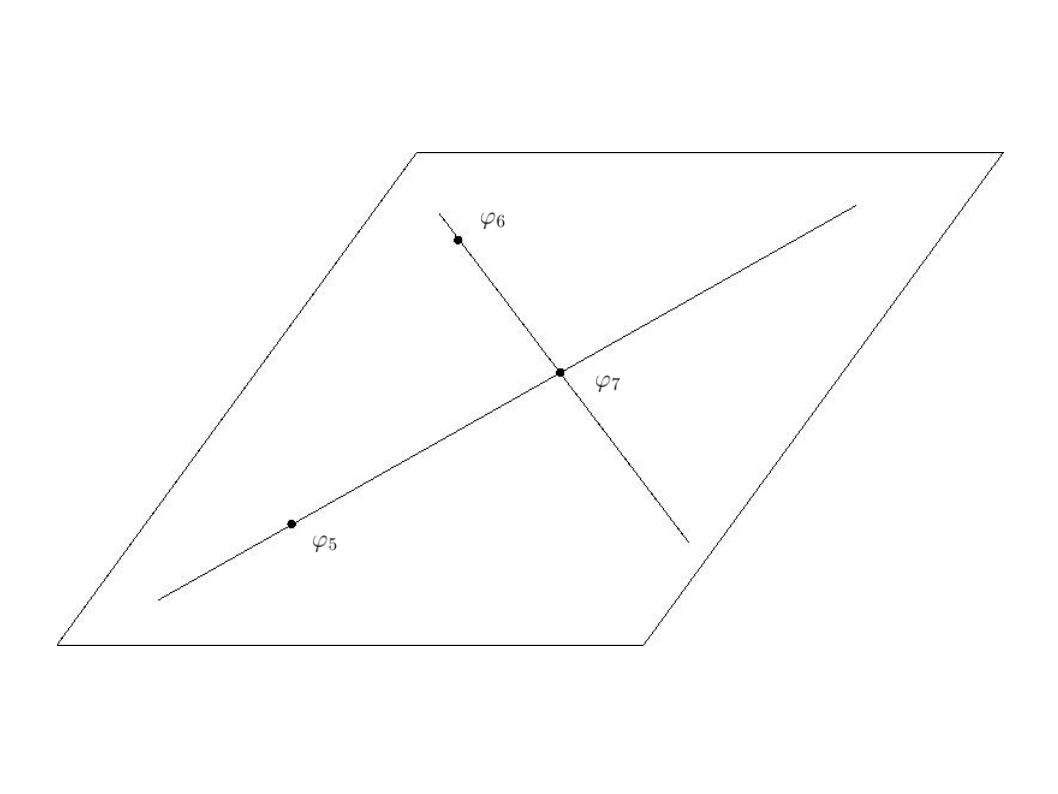}
\end{center}
\end{figure}

\subsection{$\pi\cap\cQ$ is a pair of lines}\label{pair(4,3)}

We call $p$ the intersection point of the two lines and we assume that $\f_7=p$, so that $\f_7$ has rank 2. Notice that $\f_5$, $\f_6$ and $\f_7$ span $\pi$, thus they can not be collinear. We change the basis in $F_1$ so that $\f_5$, $\f_6$ and $\f_7$ are as in figure (\ref{fig:1}); the three bivectors have rank 2 and give three $2-$planes
$\pi_5$, $\pi_6$ and $\pi_7$ in $F_0$. The projective lines $\ell_5$ and $\ell_6$, joining $\f_7$ with
$\f_5$ and $\f_6$ respectively, are contained in $\cQ$, but the line $r=\l \f_5,\f_6\r$ is not. This means that any
linear combination $a_5\f_5+a_7\f_7$ and $b_6\f_6+b_7\f_7$ has rank 2 ($a_5,a_7,b_6,b_7\in\bk$) while any
combination $c_5\f_5+c_6\f_6$, $c_5\cdot c_6\in\bk^*$ has rank 4. Going back to $F_0$, we get $\pi_5\cap\pi_7=\ell_1$
and $\pi_6\cap\pi_7=\ell_2$, while $\pi_5\op\pi_6=F_0$. We choose vectors $x_1$ spanning $\ell_1$ and $x_3$
spanning $\ell_2$, so that $\pi_7=\l x_1,x_3\r$; then we complete $x_1$ to a basis $\l x_1,x_2\r$ of $\pi_5$ and
$x_3$ to a basis $\l x_3,x_4\r$ of $\pi_6$. This gives the model
\begin{equation}\label{eq:4.3.3}
 \left\{
 \begin{array}{ccl}
 dx_5 & = & x_1x_2\\
 dx_6 & = & x_3x_4\\
 dx_7 & = & x_1x_3
 \end{array}
 \right.
\end{equation}


\subsection{$\pi\cap\cQ$ is a smooth conic}\label{conic(4,3)}

We call $\cC$ this conic and we choose the points $\f_5$, $\f_6$ on $\cC$. $\f_7$ is chosen as the intersection point between the tangent lines to the conic $\cC$ at $\f_5$ and $\f_6$. The bivectors $\f_5$ and $\f_6$ have rank 2, while $\f_7$ has rank 4. We denote $\pi_5$ and $\pi_6$ the planes in $F_0$ associated to $\f_5$ and $\f_6$ respectively. The projective line $\ell=\l\f_5,\f_6\r$ contains rank 4 bivectors, except for $\f_5$ and $\f_6$: any form $a\f_5+b\f_6$, $a\cdot b\neq 0$ has rank 4. We take coordinates in $F_0$ so that $\f_5=x_1x_2$ and $\f_6=x_3x_4$. Using these coordinates we can write
$$
\f_7=x_1x_3+\a x_1x_4+\b x_2x_3+\g x_2x_4=x_1(x_3+\a x_4)+x_2(\b x_3+\g x_4);
$$
consider the change of variables $y_3=x_3+\a x_4$, $y_4=\b x_3+\g x_4$; then, scaling $x_6$, one sees that the resulting model is
\begin{equation*}
 \left\{
 \begin{array}{ccl}
 dy_5 & = & y_1y_2\\
 dy_6 & = & y_3y_4\\
 dy_7 & = & y_1y_3+y_2y_4.
 \end{array}
 \right.
\end{equation*}

A generic point $\f$ in the plane $\pi$ may be written as $\f=X\f_5+Y\f_6+Z\f_7$ for $X,Y,Z\in\bk$. Then $\f$ has rank 2 if and only if $\f\wedge\f=0$. Computing, we obtain the conic $XY-Z^2=0$ in $\bP^2$. In our models, we have taken $\f_5$ and $\f_6$ as the intersection between the conic $XY-Z^2=0$ and the line $Z=0$; the points $\f_7$ has been chosen as intersection point between the two tangent lines to the conic at $\f_5$ and $\f_6$. Notice that over an algebraically closed field, all smooth conics are equivalent.


\section{Case $(4,3)$ when $\bk$ is not algebraically closed}

In this section we study the case $(4,3)$ when the ground field $\bk$ is not algebraically closed. In what follows, $\bP^n$ will always denote $\bP^n_{\bk}$.

When the plane $\pi$ is contained in the Klein quadric, the fact that $\bk$ is not algebraically closed does not matter. But it does matter when the $\pi$ cuts the Klein quadric in a (not necessarily smooth) conic. Indeed, the classification of conics over non-algebraically closed fields is nontrivial.

This section is organized as follows: first, we find a normal form for a conic in $\bP^2$. Then, according to this normal form, we show that any conic may be obtained as intersection between a plane $\pi\subset\bP^5$ and the Klein quadric $\cQ$; we also show how to recover the minimal algebra from the conic. Finally we give a criterion to decide whether two conics are isometric.\\

We start with the classification of conics. Fix a $3-$dimensional vector space $W$ over $\bk$ such that $\bP^2=\bP(W)$. If $\cC\subset\bP^2$ is a conic, taking coordinates $[X_0:X_1:X_2]$ in $\bP^2$ we can write $\cC$ as the zero locus of a quadratic homogeneous polynomial
$$
P(X_0,X_1,X_2)=\sum_{i\leq j}a_{ij}X_iX_j.
$$
To $\cC$ we may associate the quadratic form $Q$ defined on $W$ by the matrix $A=(a_{ij})$. A very well known theorem in linear algebra asserts that every quadratic form can be diagonalized by congruency. This means that there exists a basis of $W$ such that the matrix $B=(b_{ij})$ associated to $Q$ in this basis is diagonal and $B=P^tAP$ for an invertible matrix $P$. In this basis we can write the quadratic form as
$$
Q(Y_0,Y_1,Y_2)=\alpha Y_0^2-\beta Y_1^2-\gamma Y_2^2
$$
for suitable coefficients $\a$, $\b$ and $\g$ in $\bk$ ($\pm$ the eigenvalues of the matrix $A$). Suppose that $Q_1$ and $Q_2$ are two quadratic forms with associated matrices $A_1$ and $A_2$; then $Q_1$ and $Q_2$ are are \textit{isometric} if there exists a nonsingular matrix $P$ such that $A_2=P^tA_1P$. Since we exclude the case $Q\equiv 0$, we may assume $\a\neq 0$. The conic $\cC$ is the zero locus of the polynomial $\lambda Q$ for every $\lambda\in\bk^*$; multiplying $Q$ by $\a^{-1}$, we may assume that $\cC$ is given as zero locus of the polynomial
\begin{equation}\label{eq:conic_standard}
P(Y_0,Y_1,Y_2)=Y_0^2-a Y_1^2-b Y_2^2, \quad a,b\in\bk.
\end{equation}
We take this to be canonical form of a conic. Two conics written in the canonical form are isomorphic if and only if the corresponding quadratic forms are isometric. A first step in the classification is given by the \textit{rank} of the conic, which is defined as the rank of the associated symmetric matrix.
\begin{description}
 \item[rank 1] If $a=b=0$ we obtain the double line $Y_0^2=0$;
 \item[rank 2] if $b=0$ but $a\neq 0$ we obtain the ``pair of lines'' $Y_0^2-aY_1^2=0$;
 \item[rank 3] if $a\cdot b\neq 0$ we obtain the smooth conic $Y_0^2-a Y_1^2-b Y_2^2=0$.
\end{description}

\begin{lemma}
 Any conic can be obtained as the intersection in $\bP^5$ between the Klein quadric and a suitable plane $\pi$.
\end{lemma}
\begin{proof} Le $\cC$ be the conic defined by the equation $X^2-aY^2-bZ^2=0$, where $[X:Y:Z]$ are homogeneous coordinates in $\bP^2$ and $a,b\in\bk$. If we take coordinates $[X_0:\ldots:X_5]$ in $\bP^5$, the Klein quadric $\cQ$ is given by the equation $X_0X_5-X_1X_4+X_2X_3=0$. Consider in $\bP^5$ the plane $\pi\cong\bP^2$ of equations
$$
X_0-X_5=0, \quad X_1-aX_4=0 \quad \textrm{and} \quad X_2+bX_3=0.
$$
Then $\pi\cap\cQ$ is given by $X^2-aY^2-bZ^2=0$.
\end{proof}

When $\bk$ is algebraically closed (or in case every element of $\bk$ is a square), the geometry of the Klein quadric $\cQ$ and of the projective plane determines the minimal algebras. In fact, the differential $d:F_1\to\wedge^2 F_0$ gives a plane $\pi\subset\bP^5$ and proposition \ref{prop:plane_grassmannian} gives all the possible positions of $\pi$ with respect to $\cQ$. From each of these positions we have deduced the corresponding minimal algebra. As we said above, over a non algebraically closed field the classification of conics is more complicated and more care is needed. In particular, it is not anymore true that every conic has \textit{rational} points, where by rational point we mean points in $\bk$. The problem of determining which conics have rational points is equivalent to the problem of determining whether the quadratic form $Q$ associated to $\cC$ is isotropic, this is, if there exists a vector $v\in W$ such that $Q(v)=0$. If a conic defined over $\bk$ has no rational points, it might not be possible to choose representatives of rank 2 for the bivectors. Notice however that the rank 1 conic always has rational points. We need to discuss the rank 2 and rank 3 cases.

\subsection{Rank 2 conics} Any rank 2 conic can be put in the form $X^2-aY^2=0$. If $a$ is not a square in $\bk^*$ the conic has just one rational point, $p=[0:0:1]$. There is a quadratic extension $\bk'=\bk(\sqrt{a})$, such that $X^2-aY^2=(X-\sqrt{a} Y)(X+\sqrt{a}Y)=0$ and the conic splits as two intersecting lines in $\bP^2_{\bk'}$. The quadratic field extensions are parametrized by elements $a\in \bk^*/(\bk^*)^2-\{1\}$. The Galois group $\textrm{Gal}(\bk':\bk)$ permutes the two lines; the intersection point is fixed by this action, and thus already in $\bP^2$; it is the point $p$ above.

We set $F'_0=F_0\otimes\bk'$; the plane $\pi'=\pi\otimes\bk'$ is spanned by three bivectors $\f_5$, $\f_6$ and $\f_7$ in $\wedge^2F'_0$. We choose $\f_7=p$ and suppose that $\f_5$ and $\f_6$ are conjugated by the action of the Galois group. These points represent rank 2 bivectors, hence planes $\pi_5$, $\pi_6$ and $\pi_7$ in $F_0'$. We take vectors $x_1,\ldots,x_4$ so that $\pi_5=\l x_1,x_2\r$, $\pi_6=\l x_3,x_4\r$ and $\pi_7=\l x_1,x_3\r$, see (\ref{pair(4,3)}). The model over $\bk'$ is
$$
\left\{
 \begin{array}{ccl}
 dx_5 & = & x_1x_2\\
 dx_6 & = & x_3x_4\\
 dx_7 & = & x_1x_3
 \end{array}
 \right.
$$
Now write
$$
\left\{
 \begin{array}{ccl}
  x_1 &=& \sqrt{a}y_1+y_2, \\
  x_2 &=& \sqrt{a}y_3+y_4, \\
  x_3 &=& -\sqrt{a}y_1+y_2, \\
  x_4 &=& -\sqrt{a}y_3+y_4, \\
  x_5 &=& \sqrt{a}y_5+y_6, \\
  x_6 &=& -\sqrt{a}y_5+y_6, \\
  x_7 &=& -2\sqrt{a}y_7 \\
 \end{array}
 \right.
$$
where the $y_i$ are now defined over $\bk$. This gives the model
\begin{equation}\label{eq:4.3.5}
 \left\{
 \begin{array}{ccl}
 dy_5 & = & y_1y_4+y_2y_3\\
 dy_6 & = & ay_1y_3+y_2y_4\\
 dy_7 & = & y_1y_2
 \end{array}
 \right.
\end{equation}
with $a\in\bk^*/(\bk^*)^* - \{1\}$; this is canonical: two of these minimal algebras are not isomorphic over $\bk$ for different quadratic field extensions, since the equivalence would be given by a $\bk-$isomorphism, therefore commuting with the action of the Galois group. Note that for $a=1$, we recover case (\ref{eq:4.3.3}), where $dy_5+dy_6=(y_1+y_2)(y_3+y_4)$ and $dy_5-dy_6=(y_1-y_2)(y_3-y_4)$ are of rank $2$.

\subsection{Smooth conics} Let $\cC\subset\bP^2$ be a smooth conic; then $\cC$ can be written as $X^2-aY^2-bZ^2$ for suitable coefficients $a,b\in\bk^*$.

\begin{lemma}
 Let $p\in\cC$ be a rational point. Then $\cC$ is isomorphic to the projective line $\bP^1$.
\end{lemma}
\begin{proof}
 Fix a line $\bP^1\subset\bP^2$ not passing through $p$ and consider the set of lines in $\bP^2$ through $p$. Each
 line $\ell$ meets the conic in some point $p_{\ell}$ with coordinates in $\bk$. In fact, the coordinates of
 $p_{\ell}$ are given as solution to a quadratic equation with coefficients in $\bk$ and with one root in $\bk$.
 The map sending $p_{\ell}$ to the intersection of $\ell$ with the fixed projective line $\bP^1$ is defined on
 $\cC-\{p\}$, but can be extended to the whole $\cC$ by sending $p$ to the intersection of the tangent line at $p$
 with the fixed line. This map is birational, hence an isomorphism.
\end{proof}

By considering the inverse map, we see that every conic with a rational point can be parametrized by a projective line $\bP^1$; by this we mean that there exists an isomorphism $\bP^1\to\cC$ of the form
$$
[X_0:X_1]\to[q_0(X_0,X_1):q_1(X_0,X_1):q_2(X_0,X_1)]
$$
where $q_i=a_{i0}X_0^2+a_{i1}X_0X_1+a_{i2}X_1^2$ is a quadratic homogeneous polynomial. By letting the parametrization vary, we obtain all possible conics with rational points. This proves the following lemma:
\begin{lemma}
 Let $\cC$ be a conic in $\bP^2$ with one rational point. Then $\cC$ is projectively
 equivalent to the conic $\cC_0$ of equation $X^2+Y^2-Z^2$.
\end{lemma}
\begin{proof}
 It is clear that $\cC_0$ has rational points; for instance, $[1:0:1]\in\cC_0$. According to the previous
discussion, we can find a change of coordinates of $\bP^2$ sending $\cC$ to $\cC_0$.
\end{proof}

\begin{rem}
 It is well known that five points $p_1,\ldots,p_5\in\bP^2$, such that no three of them are colinear, determine a
 conic in $\bP^2$. There is a remarkable exception: the projective space $\bP^2_{\bZ_3}$ contains 13 points, but
 no matter how one chooses five of them, there will be at least three on a line. Indeed, a conic in $\bP^2_{\bZ_3}$ only
 has 4 points.
\end{rem}

The previous lemma allows us to divide conics in two classes:
\begin{itemize}
 \item conics with rational points; all of them are equivalent to $\cC_0$;
 \item conics without rational points.
\end{itemize}

A conic $\cC$ with equation $X^2-aY^2-bZ^2$ defined over $\bk$ without rational points has points in many quadratic
extension of $\bk$, for instance
\begin{itemize}
 \item in $\bk'=\bk(\sqrt{a})$, $p=[\sqrt{a}:1:0]$;
 \item in $\bk'=\bk(\sqrt{b})$, $p=[\sqrt{b}:0:1]$;
 \item in $\bk'=\bk(\sqrt{-a/b})$, $p=[0:1:\sqrt{-a/b}]$.
\end{itemize}

These quadratic extensions are not necessarily isomorphic if the square class group has more than two elements.\\

Suppose $\cC=X^2-aY^2-bZ^2$ is a conic without rational points. Then we consider a quadratic extension
$\bk'=\bk(\sqrt{a})$, where $\cC$ has rational points. Over $\bk'$, 
$$
X^2-aY^2-bZ^2=(X-\sqrt{a}Y)(X+\sqrt{a}Y)-bZ^2=\bar{X}\bar{Y}-b\bar{Z}^2.
$$
We set $\bar{\cC}\subset\bP^2_{\bk'}$ and we argue as in section (\ref{conic(4,3)}). We choose $\f_5$ and $\f_6$ on $\bar{\cC}$ conjugated under
the action of the Galois group $\textrm{Gal}(\bk':\bk)$ (notice that this action does not fix any point of $\bar{\cC}$); also,
we choose $\f_7$ as the intersection point between the tangent lines to $\bar{\cC}$ at
$\f_5$ and $\f_6$; hence $\f_7$ is already in $\bk$, thus fixed by the action of the Galois group. We can write, in
$F_0'=F_0\otimes\bk'$,
\begin{equation*}
 \left\{
 \begin{array}{ccl}
 dx_5 & = & x_1x_2\\
 dx_6 & = & x_3x_4\\
 dx_7 & = & a_{13}x_1x_3+a_{14}x_1x_4+a_{23}x_2x_3+a_{24}x_2x_4.
 \end{array}
 \right.
\end{equation*}
and consider
$$
\left\{
 \begin{array}{ccl}
  x_1 &=& \sqrt{a}y_1+y_2, \\
  x_2 &=& \sqrt{a}y_3+y_4, \\
  x_3 &=& -\sqrt{a}y_1+y_2, \\
  x_4 &=& -\sqrt{a}y_3+y_4, \\
  x_5 &=& \sqrt{a}y_5+y_6, \\
  x_6 &=& -\sqrt{a}y_5+y_6, \\
  x_7 &=& y_7 \\
 \end{array}
 \right.
$$
where the $y_i$ are defined over $\bk$. Then, if $\s$ is a generator of $\textrm{Gal}(\bk':\bk)$, $\s(x_1)=x_3$ and $\s(x_2)=x_4$. Thus
\begin{align*}
 [\s(dx_7)] & = a_{13}x_3x_1+a_{14}x_2x_3+a_{23}x_4x_1+a_{24}x_4x_2=\\
 & = -(a_{13}x_1x_3+a_{23}x_1x_4+a_{14}x_2x_3+a_{24}x_2x_4)=\\
 & = [dx_7] \Leftrightarrow a_{14}=a_{23},
\end{align*}
where the brackets denote the equivalence class of $dx_7$ in $\bP^2_{\bk'}$. Then we can write
$$
dx_7=a_{13}x_1x_3+a_{24}x_2x_4+a_{14}(x_1x_4+x_2x_3).
$$
Performing the change of variables we obtain
\begin{equation*}
 \left\{
 \begin{array}{ccl}
 dy_5 & = & y_1y_4+y_2y_3\\
 dy_6 & = & ay_1y_3+y_2y_4\\
 dy_7 & = & b_{12}y_1y_2+b_{34}y_3y_4+c(y_1y_4-y_2y_3).
 \end{array}
 \right.
\end{equation*}
If $b_{12}\neq 0$ we can substitute $y_1\mapsto y_1+\frac{c}{b_{12}}y_3$ and $y_2\mapsto y_2+\frac{c}{b_{12}}y_4$ to get rid of the term $c(y_1y_4-y_2y_3)$. Scaling $\f_7$, the model becomes
\begin{equation}\label{eq:592}
 \left\{
 \begin{array}{ccl}
 dy_5 & = & y_1y_4+y_2y_3\\
 dy_6 & = & ay_1y_3+y_2y_4\\
 dy_7 & = & y_1y_2+\alpha y_3y_4.
 \end{array}
 \right.
\end{equation}
From this model we must be able to recover the conic $\bar{X}\bar{Y}-b\bar{Z}^2=0$ in $\bP^2_{\bk'}$; take a generic $\f=[X:Y:Z]\in\bP^2$, where the reference system in $\bP^2$ is $\l dy_5,dy_6,dy_7\r$; then $\f$ has rank 2 if and only if $\f\wedge\f=0$, which gives the equation
$$
X^2-aY^2+\a Z^2=(X-\sqrt{a}Y)(X+\sqrt{Y})+\a Z^2;
$$
this must be equal to $\bar{X}\bar{Y}-b\bar{Z}^2=0$, which forces $\a=-b$. Finally, the model is
\begin{equation}\label{eq:4.3.6}
 \left\{
 \begin{array}{ccl}
 dy_5 & = & y_1y_4+y_2y_3\\
 dy_6 & = & ay_1y_3+y_2y_4\\
 dy_7 & = & y_1y_2-by_3y_4.
 \end{array}
 \right.
\end{equation}
Going back to (\ref{eq:592}), one sees easily that if $b_{12}=0$ but $b_{34}\neq 0$, there is a change of variables that gives again
(\ref{eq:4.3.6}). If $b_{12}=b_{34}=0$, then a linear combination of $dy_5$ and $dy_7$ has rank 2, giving some point
of intersection with the conic. But this is impossible.

To sum up, the discussion in rank 1, 2 and 3, gives the following proposition:
\begin{prop}
There is a $1-1$ correspondence between minimal algebras of type $(4,3)$ over $\bk$ such that the plane $\pi$ determined by the differential $d: F_1\to\wedge^2F_0$ is not contained in the Klein quadric $\cQ$ and the set of conics in $\bP^2_{\bk}$.
\end{prop}

The last step is to give a criterion to say when two
conics are equivalent. If the conic has rational points, it can be put in the form $\cC_0$ under a suitable change
of variables. So we assume that the conic has no rational points. As we remarked above, two conics are equivalent if and only if the
corresponding quadratic forms are isometric up to a scalar factor (which allows to write the conic in the normal form (\ref{eq:conic_standard}). The problem of establishing when two quadratic forms are isometric is
quite complicated and a complete answer requires a lot of algebra. Here we need an answer only the $3-$dimensional case. We refer to \cite{O'M} for all the details.\\

Let $W$ be a vector space of dimension 3 over $\bk$. A quadratic form on $W$ is \textit{regular} if its matrix in
any basis is nonsingular. Equivalently, if the associated conic $\cC\subset\bP^2$ is smooth.

\begin{teo}\label{thm:1}
 Let $Q_1$ and $Q_2$ be two regular quadratic forms on $W$. Then $Q_1$ and $Q_2$ are isometric if and only if
 $$
 \det(Q_1)=\det(Q_2) \qquad \textrm{and} \qquad S(Q_1)\sim S(Q_2).
 $$
\end{teo}

The \textit{determinant} of a quadratic form $Q$ is the determinant of any matrix representing $Q$;
it is well defined as an element of $\bk^*/(\bk^*)^2$. $S(Q)$ denotes the \textit{Hasse algebra} of $Q$ and $\sim$ denotes
similarity. We define the Hasse algebra and explain what it means for two Hasse algebras to be similar.

Let $V$ be a $4-$dimensional vector space over $\bk$ and let $a,b$ be two nonzero scalars. We fix a basis
$\{\mathbf{1},x_1,x_2,x_3\}$ of $V$ and define a multiplication on these basis elements according to the rules of table \ref{mult-quat}. This multiplication is extended to the whole algebra using linearity. We denote this algebra by $(a,b)$ and call it \textit{quaternion algebra}. When $\bk=\bR$ and $(a,b)=(-1,-1)$, we obtain the usual Hamilton quaternions $\bH$. As another example, we can take the algebra $M_2(\bk)$ of $2\times 2$ matrices with entries in $\bk$. It is easy to see that $M_2(\bk)\cong(1,-1)$.\\

\begin{table}[h!]
\caption{Multiplication table}\label{mult-quat}
\begin{tabular}{|c|c|c|c|c|}
\hline
 & $\mathbf{1}$ & $x_1$ & $x_2$ & $x_3$\\
\hline
$\mathbf{1}$ & $\mathbf{1}$ & $x_1$ & $x_2$ & $x_3$\\
\hline
$x_1$ & $x_1$ & $a\mathbf{1}$ & $x_3$ & $ax_2$\\
\hline
$x_2$ & $x_2$ & $-x_3$ & $b\mathbf{1}$ & $-bx_1$\\
\hline
$x_3$ & $x_3$ & $-ax_2$ & $bx_1$ & $-ab\mathbf{1}$\\
\hline
\end{tabular}
\end{table}

Quaternion algebras have the following properties:
\begin{enumerate}
 \item $(1,a)\cong(1,-1)\cong(b,-b)\cong(c,1-c)$, where $c\neq 1$;
 \item $(b,a)\cong(a,b)\cong(a\lambda^2,b\mu^2)$ for $\lambda,\mu\in\bk$;
 \item $(a,ab)\cong(a,-b)$;
 \item $(a,b)\otimes_{\bk}(a,c)=(a,bc)\otimes_{\bk}(1,-1)$.
\end{enumerate}
We may write a quaternion $q\in (a,b)$ as $\xi_0\mathbf{1}+\xi_1x_1+\xi_2x_2+\xi_3x_3$ with $\xi_i\in\bk$; its conjugate is $\bar{q}=\xi_0\mathbf{1}-\xi_1x_1-\xi_2x_2-\xi_3x_3$. The \textit{norm} of a quaternion $q$ is
$$
N(q)=q\bar{q}=\xi_0^2-a\xi_1^2-b\xi_2^2+ab\xi_3^2.
$$
The elements of $(a,b)^0=\l x_1,x_2,x_3\r\subset(a,b)$ are called purely imaginary quaternion. The norm on $(a,b)^0$ is the restriction of the norm on $(a,b)$ and is given by
$$
N(q^0)=-a\xi_1^2-b\xi_2^2+ab\xi_3^2.
$$
Therefore $N:(a,b)^0\to\bk$ defines a quadratic form on $(a,b)^0\cong\bk^3$ and hence a conic in $\bP((a,b)^0)$. Since we can multiply a conic by any $\lambda\in\bk^*$, we see that $-a\xi_1^2-b\xi_2^2+ab\xi_3^2$ is equivalent to $-b\xi_1^2-a\xi_2^2+\xi_3^2$ (multiplying it by $ab$); this is the normal form of a conic we found at the beginning of this section. This explains the relation between quaternion algebras and plane conics. We associate to the conic $-b\xi_1^2-a\xi_2^2+\xi_3^2$ the quaternion algebra $(a,b)$.

Quaternion algebras are a special example of \textit{central simple algebras}. A central simple algebra is a finite dimensional algebra $A$ over $\bk$ with unit $\mathbf{1}_A$, satisfying two conditions
\begin{itemize}
\item the center of $A$ can be identified with $\bk$ under the inclusion $\lambda\mapsto\lambda\cdot\mathbf{1}_A$;
\item $A$ contains no two-sided ideals other than $0$ and $A$ itself.
\end{itemize}

The tensor product (over $\bk$) of two central algebras is again a central algebra. Two central algebras $A$ and $B$ are \textit{similar}, written $A\sim B$, if there exist matrix algebras $M_p(\bk)$ and $M_q(\bk)$ such that
$$
A\ot M_p(\bk)\cong B\ot M_q(\bk).
$$
Let $(\mathcal{A},\ot)/\sim$ denote the set of all central simple algebras over $\bk$; one can prove that this is indeed an abelian group, called the \textit{Brauer group} $\mathrm{Br}(\bk)$ of $\bk$. For more details about the Brauer group, see for instance \cite{Mi}. Property $(4)$ above says that $(a,b)\ot(a,c)=(a,bc)$ in $\mathrm{Br}(\bk)$ since $M_2(\bk)\cong(1,-1)$.
Also, $(a,b)\ot(a,b)=(a,b^2)=(a,1)=1$ in $\mathrm{Br}(\bk)$. This proves that quaternion algebras give elements of order 2 in the Brauer group.

Suppose that $Q$ is a quadratic form on an $n-$dimensional vector space $W$ over $\bk$. In a suitable basis, the matrix of $Q$ is $\mathrm{diag}(a_1,\ldots,a_n)$ with $a_i\in\bk \ \forall \ i$. Define $d_j=\prod_{i=1}^ja_i$. The \textit{Hasse algebra} associated to $Q$ is
$$
S(Q)=\bigotimes_{1\leq j\leq n}(a_j,d_j),
$$
where $(a_j,d_j)$ denotes a quaternion algebra. Notice that the Hasse algebra is an element of the Brauer group $\mathrm{Br}(\bk)$.

Since we are working with conics, the determinant is not an invariant; indeed, we can multiply the equation of
$\cC$ by $\lambda\in\bk^*$ so that $\det(\cC)=1$ in $\bk^*/(\bk^*)^2$. On the other hand, the Hasse algebra is an
invariant, i.e. it remains unchanged when we scale the quadratic form.
\begin{lemma}
The Hasse algebras of the quadratic forms $Q$ and $\lambda Q$ are similar.
\end{lemma}
\begin{proof}
Assume that $Q$ has been diagonalized and normalized, so that $Q=X^2-aY^2-bZ^2$. Then
$$
a_1=d_1=1, \quad a_2=d_2=-a, \quad a_3=-b, \quad d_3=ab
$$ 
and
\begin{align*}
S(Q) &= (1,1)\otimes(-a,-a)\otimes(-b,ab)\sim (-a,-1)\ot(-a,a)\ot(-b,b)\ot(-b,a) \sim \\
& \sim (-a,-1)\ot(-b,a)\sim (-1,-1)\ot(a,-1)\ot(a,-b)\sim(-1,-1)\ot(a,b).
\end{align*}
Now $\lambda Q=\lambda X^2-a\lambda Y^2-b\lambda Z^2$, with
$$
a'_1=d'_1=\lambda, \quad a'_2=-\lambda a, \quad d'_2=-\lambda^2a, \quad a'_3=-\lambda b \quad \textrm{and} \quad d'_3=\lambda^3ab.
$$ 
One gets
\begin{align*}
S(\lambda Q) &= (\lambda,\lambda)\otimes(-\lambda a,-\lambda^2a)\otimes(-\lambda b,\lambda^3ab)\cong (\lambda,\lambda)\otimes(-\lambda a,-a)\otimes(-\lambda b,\lambda ab) \sim \\
& \sim (\lambda,\lambda)\ot(\lambda,-a)\ot(\lambda,\lambda ab)\ot(-b,\lambda)\ot(-a,-a)\ot (-b,ab)\sim\\
& \sim(\lambda,\lambda^2a^2b^2)\ot(-a,-a)\ot (-b,ab)\sim (-1,-1)\ot(a,b).
\end{align*}
\end{proof}
Then we may associate to the conic $\cC$ two elements of the Brauer group: the quaternion algebra $(a,b)$ and the Hasse algebra $(-1,-1)\ot(a,b)$. Theorem \ref{thm:1} says that two conics $\cC_1$ and $\cC_2$, with equations $X^2-aY^2-bZ^2$ and $X^2-\alpha Y^2-\beta Z^2$ are equivalent if and only if $S(\cC_1)\sim S(\cC_2)$, that is, if and only if
$$
(-1,-1)\ot(a,b)\sim(-1,-1)\ot(\a,\b).
$$
Since the Brauer group is a group, this is equivalent to $(a,b)\sim(\a,\b)$. Then we see that two conics are isomorphic if and only if the corresponding quaternion algebras are isomorphic and we get as many non-isomorphic conics as non-isomorphic quaternion algebras over $\bk$. Recall that the conic determines the minimal algebra; then we have shown the following proposition:
\begin{prop}
Let $(\wedge V,d)$ be a minimal algebra of dimension 7 and type $(4,3)$ and suppose that the differential $d:F_1\hookrightarrow \wedge^2F_0$ determines a plane $\pi$ which cuts the Klein quadric $\cQ$ in a smooth conic. The number of non isomorphic minimal algebras of this type is equal to number of isomorphism classes of quaternion algebras over $\bk$.
\end{prop}

\begin{rem}
We saw above that quaternionic algebras over $\bk$ define order two elements in the Brauer group $\textrm{Br}(\bk)$. The converse is partially true. In fact, Merkurjev (\cite{Me}) proves that any element of order two in the Brauer group is equal (in the Brauer group) to a product of quaternion algebras. To avoid technicalities, we prefer to state the result in term of quaternion algebras.
\end{rem}

\begin{teo}\label{thm:2}
 A quadratic form $Q$ is isotropic if and only if $S(Q)\sim(-1,-1)$.
\end{teo}

\subsection{Examples} We end this section with some examples.

Assume $\bk=\bR$; in the rank 2 case we have two conics, $X^2-Y^2$ and $X^2+Y^2$. The first is the product of two
real lines, and the second one is the product of two imaginary lines and gives the model
(\ref{eq:4.3.5}) with $a=-1$. For the rank 3 case, we use the fact that $\mathrm{Br}(\bR)\cong\bZ_2$, generated by the quaternion algebra $(-1,-1)$. We get two quadratic forms,
$Q_0=X^2+Y^2+Z^2$, which is not isotropic, and $Q_1=X^2+Y^2-Z^2$, which is isotropic. The last case has already been studied, while the first one gives the model (\ref{eq:4.3.6}) with $a=b=-1$. The Hasse algebras are $S(Q_0)\sim(1,1)$ and $S(Q_1)\sim(-1,-1)$.\\

Suppose $\bk=\bF_{p^n}$ is a finite field. It is possible to show (see for instance \cite{Se}) that any quadratic
form over a $3-$dimensional vector space over $\bF_{p^n}$ is isotropic (indeed, the Brauer group of any finite field is trivial). Then any
smooth conic in $\bP^2$ has rational points: when the conic is smooth, there is no new minimal
algebra with respect to the algebraically closed case. On the other hand, in the rank 2 case we obtain the model
(\ref{eq:4.3.5}), with $a\in\bk^*/(\bk^*)^2-\{1\}$. Since for finite fields $|\bk^*/(\bk^*)^2|=2$, we get only one further
minimal algebra.\\

Finally, we treat the case $\bk=\bQ$, which is very relevant on the rational homotopy side. The rank 2 case
is straightforward: we get as many models as elements in $\bQ^*/(\bQ^*)^2$, all of them of the form
(\ref{eq:4.3.5}). In the rank 3 case, we have the following exact sequence
$$
0\rightarrow \mathrm{Br}(\bQ)\rightarrow\bigoplus_{p\in\cP}\mathrm{Br}(\bQ_p)\rightarrow\bQ/\bZ\rightarrow 0,
$$
where $\cP=\{2,3,5,\ldots,\infty\}$ is the set of all prime numbers and $\infty$, $\bQ_p$ is the field of $p-$adic numbers and, by definition,
$\bQ_{\infty}=\bR$. We remarked above that quaternion algebras are related to the $2-$torsion in the Brauer group of $\bk$. Every $p-$adic field $\bQ_p$ has two non isomorphic quaternionic algebras, one isotropic and one non isotropic. The above exact sequence shows that $\bQ$ has an infinite number of non-isomorphic quaternionic algebras. 

We give another method to establish whether a conic $\cC$ defined over $\bQ$ has rational points or not; we refer to \cite{Se} for further details. Since $\bQ\subset\bQ_p$ for every $p\in\cP$, $\cC$ can be interpreted as a quadratic form over $\bQ_p$. If $\cC$ is a conic in $\bP^2_{\bQ_p}$, zero locus of $X^2-aY^2-bZ^2$ with $a,b\in\bQ$, we define its \textit{Hilbert symbol} as
$$
(a,b)_p=\left\{
\begin{array}{cl}
1 & \mathrm{if} \ X^2-aY^2-bZ^2 \ \textrm{is isotropic}\\
-1 & \textrm{otherwise.}
\end{array}
\right.
$$
The Hilbert symbol satisfies the following properties (compare with the properties of the Hasse algebra):
\begin{enumerate}
 \item $(a,b)_p=(b,a)_p$ and $(a,c^2)_p=1$;
 \item $(a,-a)_p=1$ and $(a,1-a)_p=1$ ($a\neq 0,1$);
 \item $(aa',b)_p=(a,b)_p(a',b)_p$ (bilinearity);
 \item $(a,b)_p=(a,-ab)_p=(a,(1-a)b)_p$.
\end{enumerate}
It can be easily computed according to the following rules; suppose first that $p\neq\infty$; write $a=p^{\alpha} u$, $b=p^{\beta} v$ for $\a,\b\in\bZ$ and $u,v\in\bQ_p^*$; then
$$
(a,b)_p=\left\{
\begin{array}{cl}
(-1)^{\a\b\varepsilon(p)}\left(\frac{u}{p}\right)^{\b}\left(\frac{v}{p}\right)^{\a} & \mathrm{if} \ p\neq 2\\
\\
(-1)^{\varepsilon(u)\varepsilon(u)+\a\omega(v)+\b\omega(u)} & \mathrm{if} \ p=2.
\end{array}
\right.
$$
where $\varepsilon(p)$ is the class $\frac{p-1}{2}\mod 2$ and $\omega(u)$ is the class $\frac{u^2-1}{8}\mod 2$. The
case $p=\infty$ is straightforward: $(a,b)_{\infty}=-1$ if and only if the conic is $X^2+Y^2+Z^2$. The Hasse-
Minkovski theorem says that a quadratic form defined over $\bQ$ is isotropic if and only if it is isotropic over
$\bQ_p$ for every $p\in\cP$.\\

\section{Classification}
In this section we prove the main theorem and display the results in two tables.

\begin{proof} (\textit{of the main theorem})
The theorem is a consequence of the case by case analysis of the previous sections. Case $(6,1)$ gives 3 isomorphism classes. Case $(5,2)$ gives $5+(r-1)$ isomorphism classes, where $r=|\bk^*/(\bk^*)^2|$. Finally, case $(4,3)$ gives $5+(r-1)+(s-1)$ isomorphism classes, where $s$ is the number of isomorphism classes of quaternion algebras over $\bk$. Summing the three numbers yields the thesis.
\end{proof}

%

The next table contains a list of $7-$dimensional minimal algebras of length 2, generated in degree 1, over any field $\bk$.

\begin{table}[h!]
\caption{Minimal algebras of dimension 7 and length 2 over any field}\label{table:1}
\begin{tabular}{|c|c|c|c|}
\hline
 $(f_0,f_1)$   & $dx_5$ & $dx_6$ & $dx_7$\\
\hline
(6,1)  & 0 & 0 & $x_1x_2$\\
\hline
     & 0 & 0 & $x_1x_2+x_3x_4$\\
\hline
     & 0 & 0 & $x_1x_2+x_3x_4+x_5x_6$\\
\hline
(5,2)  & 0 & $x_1x_2$ & $x_1x_3$\\
\hline
& 0 & $x_1x_2$ & $x_3x_4$\\
\hline
& 0 & $x_1x_2$ & $x_1x_3+x_2x_4$\\
\hline
& 0 & $x_1x_2$ & $x_1x_3+x_4x_5$\\
\hline
& 0 & $x_1x_2+x_3x_4$ & $x_1x_3+x_2x_5$\\
\hline
& 0 & $x_1x_3+ax_2x_4$ & $x_1x_4+x_2x_3$\\
\hline
(4,3) & $x_1x_2$ & $x_1x_3$ & $x_1x_4$\\
\hline
& $x_1x_2$ & $x_1x_3$ & $x_2x_3$\\
\hline
& $x_1x_2$ & $x_1x_3$ & $x_1x_4+x_2x_3$\\
\hline
& $x_1x_2$ & $x_3x_4$ & $x_1x_3$\\
\hline
& $x_1x_2$ & $x_3x_4$ & $x_1x_3+x_2x_4$\\
\hline
& $x_1x_4+x_2x_3$ & $ax_1x_3+x_2x_4$ & $x_1x_2$\\
\hline
& $x_1x_4+x_2x_3$ & $ax_1x_3+x_2x_4$ & $x_1x_2-bx_3x_4$\\
\hline
\end{tabular}
\end{table}

The minimal algebras in lines 9 and 15 depend on a parameter $a\in\bk^*/(\bk^*)^2-\{1\}$. The minimal algebra in line 16 depend on two parameters $a,b$. To them one associates the quaternion algebra $(a,b)$ as explained above; then the pair $(a,b)$ varies in $\bk^*\times\bk^*$ and two pairs give the same minimal algebra if and only if the corresponding quaternion algebras are isomorphic.

Next we collect the results on minimal algebras over $\bR$. Each of these minimal algebras is defined over
$\bQ$; accordingly, the corresponding nilpotent Lie algebra $\fg$ has rational structure constants and Mal'cev
theorem implies that there exists a nilmanifold associated to each of these algebras. We use Nomizu
theorem to compute the real cohomology (and the real homotopy type) of this nilmanifold. The last four columns
display the Betti numbers of the nilmanifold. The last column gives a labelling of the minimal algebras when interpreted as a Lie algebras. The notation refers to \cite{BM}. This list coincides with the one contained in the paper \cite{CF}, which, in turn, relies on \cite{Gon}.

\begin{table}[h]
\caption{Minimal algebras of dimension 7 and length 2 over $\bR$}\label{table:2}
\begin{tabular}{|c|c|c|c|c|c|c|c|c|c|c|c|}
\hline
 $(f_0,f_1)$   & $dx_5$ & $dx_6$ & $dx_7$ & $b_1$ & $b_2$ & $b_3$ & $\sum_ib_i$ & $\fg$\\
\hline
(6,1)  & 0 & 0 & $x_1x_2$ & 6 & 16 & 25 & 71 & $L_3\op A_4$\\
\hline
     & 0 & 0 & $x_1x_2+x_3x_4$ & 6 & 14 & 19 & 61 & $L_{5,1}\op A_2$\\
\hline
     & 0 & 0 & $x_1x_2+x_3x_4+x_5x_6$ & 6 & 14 & 14 & 56 & $L_{7,1}$\\
\hline
(5,2)  & 0 & $x_1x_2$ & $x_1x_3$ & 5 & 13 & 21 & 59 & $L_{5,2}\op A_2$\\
\hline
& 0 & $x_1x_2$ & $x_3x_4$ & 5 & 12 & 18 & 54 & $L_3\op L_3\op A_1$\\
\hline
& 0 & $x_1x_2$ & $x_1x_3+x_2x_4$ & 5 & 12 & 18 & 54 & $L_{6,1}\op A_1$\\
\hline
& 0 & $x_1x_2$ & $x_1x_3+x_4x_5$ & 5 & 10 & 16 & 48 & $L_{7,2}$\\
\hline
& 0 & $x_1x_2+x_3x_4$ & $x_1x_3+x_2x_5$ & 5 & 9 & 15 & 45 & $L_{7,3}$\\
\hline
& 0 & $x_1x_3-x_2x_4$ & $x_1x_4+x_2x_3$ & 5 & 12 & 18 & 54 & $L_{6,2}\op A_1$\\
\hline
(4,3) & $x_1x_2$ & $x_1x_3$ & $x_1x_4$ & 4 & 12 & 18 & 52 & $L_{7,4}$\\
\hline
& $x_1x_2$ & $x_1x_3$ & $x_2x_3$ & 4 & 11 & 20 & 52 & $L_{6,4}\op A_1$\\
\hline
& $x_1x_2$ & $x_1x_3$ & $x_1x_4+x_2x_3$ & 4 & 11 & 17 & 49 & $L_{7,5}$\\
\hline
& $x_1x_2$ & $x_3x_4$ & $x_1x_3$ & 4 & 11 & 16 & 48 & $L_{7,6}$\\
\hline
& $x_1x_2$ & $x_3x_4$ & $x_1x_4+x_2x_3$ & 4 & 11 & 14 & 46 & $L_{7,7}$\\
\hline
& $x_1x_4+x_2x_3$ & $-x_1x_3+x_2x_4$ & $x_1x_2$ & 4 & 11 & 16 & 48 & $L_{7,8}$\\
\hline
& $x_1x_4+x_2x_3$ & $-x_1x_3+x_2x_4$ & $x_1x_2+x_3x_4$ & 4 & 11 & 14 & 46 & $L_{7,9}$\\
\hline
\end{tabular}
\end{table}

\end{document}